\documentclass{amsart}

\usepackage{amssymb}
\usepackage{amsmath}
\usepackage{graphics}
\newtheorem{theorem}{Theorem}[section]
\newtheorem{proposition}[theorem]{Proposition}

\newtheorem{lemma}[theorem]{Lemma}
\newtheorem{claim}{Claim}

\newtheorem{question}[theorem]{Question}
\theoremstyle{definition}
\newtheorem{definition}[theorem]{Definition}
\newtheorem{example}[theorem]{Example}

\theoremstyle{remark}
\newtheorem{remark}[theorem]{Remark}

\numberwithin{equation}{section}

\newcommand{\F}{\ensuremath{\mathcal{F}}}

\newcommand{\rank}{\ensuremath{\mathrm{rank}\ }}

\newcommand{\RR}{\mathbb R}

\newcommand{\zermelo}{\ensuremath{\mathtt{Z}}}

\newcommand{\metriczermelo}{\ensuremath{ \mathtt{h}   } }

\newcommand{\btransnormal}{\ensuremath{\mathfrak{b}}}

\begin{document}


\title[Finsler transnormal functions]{On Finsler  transnormal functions }

\author[M. M. Alexandrino] {Marcos M. Alexandrino}

\author[B. O. Alves]{Benigno O. Alves}

\author[H. R. Dehkordi]{Hengameh R. Dehkordi  }

\address{Marcos M. Alexandrino and  Hengameh R. Dehkordi \\ 
Instituto de Matem\'{a}tica e Estat\'{\i}stica\\
Universidade de S\~{a}o Paulo \\
 Rua do Mat\~{a}o 1010, 05508 090 S\~{a}o Paulo, Brazil}
\email{m.alexandrino@usp.br}
\email{hengamehraeesi@gmail.com}

\address{Benigno O. Alves\\
Departamento de Matem\'{a}tica CCET-  UFMA \\ 
Cidade Universit\'{a}ria Dom Delgado\\
Av. dos Portugueses, 1966, Bacanga - CEP 65080-805
S\~{a}o Lu\'{\i}s - MA, Brazil}
\email{gguialves@hotmail.com}

\thanks{
The first author was supported by Funda\c{c}\~{a}o de Amparo a Pesquisa
do Estado de S\~{a}o Paulo-FAPESP (Tematico 2016/23746-6). The second and 
third authors were supported by CNPq and Capes (PhD fellowships). }

\subjclass[2000]{Primary 53C12, Secondary 58B20}


\keywords{Finsler foliations, transnormal functions,  Morse-Bott functions}

\begin{abstract}

In this note we discuss a few  properties of transnormal Finsler functions, i.e., the 
natural generalization of distance functions and isoparametric Finsler functions.
In particular, we prove that critical level sets of an analytic transnormal function are submanifolds,
and the partition of $M$ into level sets is a Finsler partition, when the function is defined on
a compact analytic manifold $M$.

\end{abstract}


\maketitle

\section{Introduction}

Let $(M,F)$ be a  forward complete Finsler manifold.  
A function $f:M\to \mathbb{R}$ 
is called $F$-\emph{transnormal function} if
$F(\nabla f)^{2}= \btransnormal(f)$ for some continuous function $\btransnormal$. 

Transnormal functions on Riemannian manifolds have been focus of researchers in the last decades. 
In particular, if $\btransnormal \in C^{2}(f(M))$  then the level sets of $f$ are leaves of the so called singular Riemannian
foliation  and the regular level sets are equifocal hypersurfaces; see e.g, \cite{Th}, \cite{QMWang} and  \cite[Chapter 5]{AlexBettiol}.

In Finsler geometry, the study of transnormal functions has just begun, see \cite{He-Yin-Shen}  but there are already some
interesting applications in wildfire modeling, see \cite{Markvorsen}.

The most natural example of a transnormal function on a Finsler space is the distance function on a Minkowski space. More precisely consider
a Randers Minkowski space $(V,\zermelo)$  and define $f(x):=d(0,x)$. 
It is well known that in this example $\btransnormal=1$, see \cite[Lemma 3.2.3]{Zhongmin-Shen}. 
And already here one can see a phenomenon that does not exist in the Riemannian case. 
The regular level set $f^{-1}(c_{1})$ is  forward parallel  to $f^{-1}(c_{2})$ if $c_{1}< c_{2}$ but
 $f^{-1}(c_{2})$ is not forward parallel to $f^{-1}(c_{1})$  and 
hence the partition $\F=\{ f^{-1}(c) \}_{c\in (0,\infty)}$ is not a  Finsler partition of $V\setminus 0$, 
recall basic definitions and examples in Sections \ref{section-basic-definitions} and  \ref{section-examples} respectively.

The above  observation leads  us to the 3 natural questions  we want to address here.  

\begin{question}
\label{question-1}
If $f$ is a transnormal function  on a forward complete Finsler manifold $(M,F),$ 
is the (regular) level set    $f^{-1}(c)$ forward parallel to the (regular) level set  $f^{-1}(d)$ when $c<d$? In this case, is
 the distance from $f^{-1}(c)$  to $f^{-1}(d)$   described by $\int_{c}^{d}\frac{1}{\sqrt{\btransnormal(s)}} ds$
as it was in the Riemannian case?
\end{question}

A positive answer to Question \ref{question-1} has already been given at \cite{He-Yin-Shen}. 
In Section \ref{section-question1} we review this fact, see Proposition \ref{proposition-forward-equidistant}. 

\begin{question}
\label{question-2}
Under which conditions the level sets of a transnormal  function are parallel to each other, i.e.,
 $\F=\{ f^{-1}(c) \}$  is a Finsler partition? 
\end{question}

As we have remarked before, level sets of transnormal functions, do not need to be equidistant, and hence some hypothesis is needed to assure equidistance between the level sets, i.e, that the  partition $\F=\{ f^{-1}(c) \}_{c\in f(M)}$ is a Finsler partition. 
The next result approaches  Question \ref{question-2} and will be discussed in Section \ref{section-question2}.  

\begin{theorem}
\label{finsler-partition}
Let $(M,F)$ be a connected, compact analytic Finsler manifold 
	and $f:M\to \mathbb{R}$ a $F$-transnormal and analytic function with $f(M)=[a,b]$. 
	Suppose that the level sets are connected and $a$ and $b$ are the only singular values at $[a,b]$. Then
	\begin{enumerate}
	\item[(a)] the critical level sets $f^{-1}(a)$ and $f^{-1}(b)$ are submanifolds. 
	\item[(b)]  The level sets are equidistant to each other, i.e, $\F=\{ f^{-1}(c) \}_{c\in [a,b]}$ is a Finsler partition. In particular for each regular value $c$, $f^{-1}(c)$ is a future and past cylinder over each singular level set.
	\end{enumerate}	
\end{theorem}
\begin{remark}
It follows from the above theorem, that the regular level set $f^{-1}(c)$ are \emph{equifocal hypersurface}, 
recall definition at \cite{AlexAlvesJavaloyes}. 
\end{remark}

Finally, inspired by \cite{AlexAlvesJavaloyes} it is also natural to ask.

\begin{question}
\label{question-3}
Under which conditions the level sets of a transnormal function on a Finsler manifold are level sets of
singular Riemannian foliation for some Riemannian metric? 

\end{question}

In section \ref{section-question3} we approach this  question 
by using the  result  of C. Qian, Z. Tang \cite{qian2015isoparametric} about Morse-Bott functions.

\

\emph{Acknowledgements}
We thank prof. Miguel Angel Javaloyes for useful suggestions. This note is based on part of H.~ R.~ Dehkordi's PhD \cite{Dehkordi}.

\section{Preliminaries}
\label{section-basic-definitions}

In this section we fix some notations and briefly review a few facts
about Finsler geometry and Finsler partitions which will be used in this note. For more
details see  \cite{Zhongmin-Shen},  \cite{AlexAlvesJavaloyes} and \cite{AlvesJavaloyes}.

\subsection{Finsler metrics}

Let $V$ be a vector space and $F:V\rightarrow [0,+\infty)$ a function. 
We say that $F$ is a {\it Minkowski norm} and $(V,F)$ is a \emph{Minkowski space} if:
\begin{enumerate}
\item[(a)] $F$ is smooth on $V \setminus  \{0 \}$,
\item[(b)] $F$ is positive homogeneous of degree 1, that is 
$F(\lambda v)=\lambda\, F(v)$ for
every $v\in V$ and $\lambda>0$,
\item[(c)] for every $v\in V \setminus \{ 0\}$, the \emph{fundamental tensor} of $F$ 
defined as
\begin{equation}
\label{fundtensor}
 g_{v}(u,w)=\frac{1}{2}\frac{\partial^{2}}{\partial t\partial  s} F^{2}(v+tu+sw)|_{t=s=0}
 \end{equation}
for any $u,w\in V$ is a positive-definite bilinear  symmetric form. 
\end{enumerate} 
 Now let us consider  a manifold $M$.  
We say that a  function $F:TM\to [0,+\infty)$ is a \emph{Finsler metric} 
if  $F$ is smooth on $TM\setminus \bf 0$,
and for every $p\in M$, $F_p=F|_{T_pM}$ is a Minkowski norm on $T_pM$.

\begin{lemma}
\label{lemma-properties-fundamental-tensor}
The fundamental tensor $g_v$ fulfills the following properties:
\begin{enumerate}
\item[(a)] $g_{\lambda v}= g_{v}$ for $\lambda>0$ 
\item[(b)] $g_{v}(v,v)=F^{2}(v)$.
\item[(c)] $g_{v}(v,u)=\frac{1}{2} \frac{\partial}{\partial z} F^{2}(v+ zu)|_{z=0}.$
\end{enumerate}
\end{lemma}
Item (c) above implies that the  Legendre
 transformation  $\mathcal{L}:TM\setminus{\bf 0}\mapsto TM^{*}\setminus{\bf 0}$ 
associated with   $\frac{1}{2}F^{2}$ can be computed as $\mathcal{L}(v)= g_{v}(v,\cdot).$
Now for a Finsler manifold $(M,F)$ and a  
smooth function $f:M\to \mathbb{R}$  we can define $\nabla f$ (the \emph{gradient with respect to $F$}) 
as $\nabla f=\mathcal{L}^{-1}df,$ i.e.,  $df (\cdot)=g_{\nabla f}(\nabla f, \cdot).$ 

 A smooth function $f:M\rightarrow \RR$ on a Finsler manifold $(M,F)$ 
is called $F$-{\it transnormal function} if  there exists another continuous  real valued function 
$\btransnormal :f(M)\rightarrow \RR$ such that  $F(\nabla f)^2=\btransnormal (f)$.

We also need to recall the definition of Chern connection, Cartan tensor and geodesics. 

\begin{lemma}[Chern's connection]
Given a  vector field $V$ without singularities on an open set $U\subset M$ there exists a unique affine connection $\nabla^{V}$ on $U$
that satisfies the following properties:
\begin{enumerate}
\item $ \nabla^{V}_{X}Y-\nabla^{V}_{Y}X=[X,Y]$
for every vector field $X$ and $Y$ on $U$,
\item $X\cdot g_{V}(Y,Z)=g_{V}(\nabla^{V}_{X}Y,Z)+g_{V}(Y,\nabla^{V}_{X}Z)+2C_{V}(\nabla^{V}_{X}V,Y,Z), $
\end{enumerate}
where $X$, $Y$, and $Z$ are vector fields on $U$  and $C_{V}$ is the Cartan tensor, i.e.,
$$C_{v}(w_{1},w_{2},w_{3}):= \frac{1}{4} \frac{\partial^{3}}{\partial s_{3}\partial s_{2}\partial s_{1}} F^{2}(v+ \sum_{i=1}^{3} s_{i}w_{i})|_{s_{1}=s_{2}=s_{3}=0}$$ 
for $v\in TM \setminus{\bf 0}$ and $w_{i}\in T_{\pi(v)}M,$ where $\pi:TM\to M$ is the canonical projection. 
\end{lemma}

Among other properties Cartan tensor satisfies:
\begin{equation}
\label{eq-properties-Cartan-tensor}
 C_{v}(v,w_{1},w_{2})=C_{v}(w_{1},v,w_{2})=C_{v}(w_{1},w_{2},v)=0 
\end{equation}

Let $\gamma:I\subset \mathbb{R}\to M$ be a piece-wise smooth curve.
As usual we can induce on the pullback bundle $\gamma^{*}(TM)$ over $I$
a covariant derivative $\frac{\nabla^{\gamma'}}{dt}$. A curve $\gamma$ is called \emph{geodesic}
if $\frac{\nabla^{\gamma'}}{dt} \gamma'(t)=0, \forall t\in I$.
Like in Riemannian geometry, for every vector $v\in TM$, there exists a unique maximal geodesic 
$\gamma_v:(a,b)\rightarrow M$ such that $\gamma_{v}'(0)=v$ 
and one can define the exponential map in an open subset 
$\mathcal U\subset TM$ for those vectors $v$ such that the maximal interval of definition 
$(a,b)$ of $\gamma_v$ includes the value $1$. 
Then $\exp:\mathcal U\rightarrow M$ is defined as $\exp(v)=\gamma_v(1)$. This map is smooth away 
from the zero section and  $C^1$ on the zero section.
Define the Finsler distance $d(p,q)$  as the infimum of the lenghts of all 
piecewise smooth curve joining $p$ to $q$, where
the lenght of a curve $\alpha:[a,b]\to M$ (with $p=\alpha(a)$ and $q=\alpha(b)$)  is 
defined as $\mathtt{l}(\alpha)=\int_{a}^{b} F(\alpha'(t)) dt.$ 
Note that $d(p,q)$ may not be equal to  $d(q,p)$. 
Geodesics locally minimize $\mathtt{l}$
among piecewise smooth curves, and hence locally realize the distance.   
More generally, a geodesic $\gamma$  minimizes (in some interval) 
the distance with a submanifold   $P$ if it is \emph{orthogonal} to $P$, i.e., 
if $g_{\gamma'(0)}(\gamma'(0),u)=0$ for all $u\in T_{\gamma(0)}P$.

\subsection{Finsler partition}

Let  $(M,F)$ be a  Finsler manifold. A partition $\F=\{L\}$ of $M$ into connected immersed smooth  submanifolds (the \emph{leaves}) is called \emph{a  Finsler partition}  if  each  geodesic  
$\gamma:(a,b)\rightarrow M$, with $0\in (a,b)\in\mathbb{R}$ 
 orthogonal to the leaf  $L_{\gamma(0)}$   is \emph{horizontal}, i.e., is orthogonal to each leaf it meets. 
 In addition a partition is called a \emph{singular foliation} if for each  $v\in T_pL_p$ there exists a smooth vector field $X$ tangent to the leaves so that $X(p)=v$.

Given a \emph{plaque} $P_{q}$ of a leaf $L$ (i.e, a `` small" relatively compact neighborhood of $q$ in  $L$)  
the set of all (non zero) orthogonal vectors to $P_{q}$ at $q$, denoted as $\nu_{q}P_{q}$ 
is called the \emph{orthogonal cone} and, as the name suggests, 
it  is not always a subspace (without zero) but a cone.

Recall that $U^{+}$ is called a \emph{(future) tubular neighborhood} (of radius $\epsilon$) of the plaque $P_{q}$
if $\exp$ sends $\nu(P_q)\cap F^{-1}((0,\epsilon))$
diffeomorphically to  $U^{+}\setminus P_{q}$,  and all the orthogonal unit speed geodesics from the plaque minimize the distance from the plaque,  at least in the interval $[0,\epsilon]$.   
If we restrict the exponential map $\exp$ to the $\epsilon$-orthogonal cone bundle 
$\nu^{\epsilon}(P_{q}):=\nu(P_{q})\cap F^{-1}(\epsilon)$, then $\exp$ sends $\nu^{\epsilon}(P_{q})$ to the so-called 
\emph{future} cylinder  $\mathcal{C}^{+}_{r}(P_{q})$. Alternatively, defining $f_+:U^{+}\to [0,+\infty)$ as the (future) distance  
$f_{+}(x):=d(P_{q},x)$, we can define $\mathcal{C}^{+}_{r}(P_{q}):=f_{+}^{-1}(r).$ 
Similarly one can define a \emph{past} (or reverse) tubular neighborhood $U^{-}$, and past cylinder $C^{-}_{r}(P_{q})$ considering the reverse metric $F^{-}(v):=F(-v)$; see e.g. \cite[Section 1.5]{Dehkordi}.  

\begin{definition}
We will say that a partition $\F$ is \emph{ locally forward (resp. backward) equidistant} if given a plaque $P_q$, a  future tubular neighborhood $U^{+}$ (resp. a reverse or past tubular neighborhood $U^{-}$) of $P_q$ and a point $x\in U^{+}$ (resp. $x\in U^{-}$) which belongs to the future cylinder  $\mathcal{C}^{+}_{r_{1}}(P_{q})$ (resp. the past cylinder $\mathcal{C}^{-}_{r_{2}}(P_{q})$),  then the  plaque
 $P_{x}\subset U^{+}$ (resp. $P_x\subset  U^{-}$)  is contained in  
$ \mathcal{C}^{+}_{r_{1}}(P_{q})$ (resp.  $\mathcal{C}^{-}_{r_{2}}(P_{q})$).
\end{definition}

\begin{lemma}[\cite{AlexAlvesJavaloyes}]
\label{lemma-equidistant}
A  partition $\F$  is Finsler if and only if its leaves are locally forward and backward equidistant.
\end{lemma}

In the particular case of a foliation of codimension 1
given by pre images of a function $f:M\to \mathbb{R}$ 
we have the following useful definition.

\begin{definition}[Forward parallel level sets]
Let $f:M\to\mathbb{R}$ be a smooth function and $f^{-1}(c_1)$ and $f^{-1}(c_2)$ 
two regular level sets, with $c_{1} < c_{2}.$  We say that 
$f^{-1}(c_{1})$ is \emph{is forward parallel to} $f^{-1}(c_{2})$
if each geodesic that starts orthogonal to $f^{-1}(c_{1})$ and meets 
 $f^{-1}(c_2)$ is orthogonal to $f^{-1}(c_{2}).$               
\end{definition}


\section{Basic remarks and examples}
\label{section-examples}

In this section we  discuss a few basic examples of transnormal functions
on Finsler manifolds  stressing differences between them and transnormal functions on Riemannian manifolds.

Along  this section we  restrict our attention to a special but important type  of Finsler metric. 
A Finsler metric  $\zermelo:TM\to [0,+\infty)$ is said to be a \emph{Randers metric}
with \emph{Zermelo Data} $(\metriczermelo,W),$  for a Riemannian metric $\metriczermelo$ 
and  smooth vector field  $W$ with $\metriczermelo(W,W)<1$  on  $M$  (the \emph{wind}), if $\zermelo$ is the solution of 
\begin{equation}
\label{eq-ex-1}
\metriczermelo(\frac{v}{\zermelo(v)}-W,\frac{v}{\zermelo(v)}-W)=1, v \in TM\setminus \bf 0.
\end{equation}
Equivalently we can define $\zermelo(v)=\alpha(v)+\beta(v)$ where $\alpha$ is a Riemannian norm and $\beta$ a 1-form 
(with $\alpha(\beta)<1$ ) both determined by  $(\metriczermelo,W);$ recall   \cite{Robles}.

\begin{lemma}
\label{lemma-gradient-finsler-riemannian}
Let $f:U\subset M\to \mathbb{R}$ be a smooth function without critical points on $U$. Let $Z$ be a Randers metric with Zermelo data 
$(h,W).$ Let $\nabla f$ and $\widetilde{\nabla} f$ be  the gradients with respect to $Z$ and $h$.   
Then
\begin{enumerate}
\item[(a)] $\frac{\| \widetilde{\nabla} f \|}{Z(\nabla f)}\Big( \nabla f-Z(\nabla f)W  \Big)= \widetilde{\nabla} f    $
\item[(b)] $Z(\nabla f)=\| \widetilde{\nabla} f \| + df(W) $
\end{enumerate}
where $\| v \|=\sqrt{h(v,v)}$. 
\end{lemma}
\begin{proof}
For a Randers metric $\zermelo$ it is well known (e.g., \cite[Cor. 4.17]{Javaloyes-Sanchez}) that:
 \begin{equation} 
\label{eq-ex-1-5}
g_{v}(v,u)=\frac{Z(v)}{\mu \alpha(v)}\Big( h(v-Z(v)W ,u) \Big) 
\end{equation}
where $\mu=1-\alpha(\beta)^{2}.$ Eq. (\ref{eq-ex-1-5}) and the definition of gradient imply
\begin{eqnarray*}
h(\widetilde{\nabla} f, u)& = & d f(u)\\
&=& g_{\nabla f}(\nabla f,u)\\
&=&\frac{Z(\nabla f)}{\mu \alpha(\nabla f)}\Big( h(\nabla f-Z(\nabla f)W ,u) \Big)
\end{eqnarray*}

Therefore, by setting $k:=\frac{Z(\nabla f)}{\mu \alpha(\nabla f)}>0 $ we have:
\begin{equation}
\label{eq-ex-2}
\widetilde{\nabla} f=k\Big( \nabla f-Z(\nabla f)W  \Big);
\end{equation}
By taking the norm $\| \cdot \|$ on both sides of Eq.~(\ref{eq-ex-2}) and replacing $v=\nabla f$ in Eq.~(\ref{eq-ex-1})   we infer
\begin{equation}
\label{eq-ex-3}
\|\widetilde{\nabla} f\|=k\| ( \nabla f-Z(\nabla f)W  )\|=k Z(\nabla f)
\end{equation}
and hence $k=\frac{\|\widetilde{\nabla} f\|}{Z(\nabla f)}$. This together with Eq.~(\ref{eq-ex-2}) finish the proof of item (a).

Item (a), Eq.~(\ref{eq-ex-1-5}) and item (b) of Lemma \ref{lemma-properties-fundamental-tensor} imply
\begin{eqnarray*}
Z^{2}(\nabla f) &=& g_{\nabla f}(\nabla f, \nabla f)\\
                &=& k h\big(\nabla f-Z(\nabla f)W,\nabla f  \big)\\
								&=& k h\big(\frac{\widetilde{\nabla}f }{k}, \frac{\widetilde{\nabla}f }{k}+Z(\nabla f)W  \big)\\
								&=& \frac{\|\widetilde{\nabla}f\|^{2} }{k}+ Z(\nabla f) df (W)
\end{eqnarray*}
The above equation finishes  the proof of item (b) because  $k=\frac{\|\widetilde{\nabla} f\|}{Z(\nabla f)}$.

\end{proof}

\begin{remark}
\label{remark-transnormal-function}
It was proved in \cite[Proposition 2.12]{AlexAlvesJavaloyes} that a partition $\F=\{ L\}$ given by a submersion 
on a Randers space $(M,\zermelo)$ with Zermelo data $(\metriczermelo, W)$ is Finsler, i.e.,  
its leaves are locally forward and backward equidistant, 
if and only if $W$ is $\F$ foliated vector field
(i.e., it projects to a vector field on the basis of the submersion)
  and  $\F$ is Riemannian with respect to $\metriczermelo$. 
This result together with  Lemma \ref{lemma-gradient-finsler-riemannian}
imply the following result: 
\emph{ Let $f:M\to \mathbb{R}$ be a $\zermelo$-transnormal function on a Randers space $(M,\zermelo)$ 
with Zermelo data $(\metriczermelo, W)$. Then the regular level sets are leaves of a Finsler partition, if and only if $W$ 
is $\F$ foliated vector field}. 
\end{remark}

\begin{example}
Let $(V,\zermelo)$ be a Randers Minkowski space with Zermelo data $(\metriczermelo, W)$ and define $f(x):=d(0,x)$. 
From  \cite[Lemma 3.2.3]{Zhongmin-Shen} we know that  $\btransnormal=1$, i.e, $f:V\setminus \{0\} \to \mathbb{R}$ is a 
$\zermelo$-transnormal function. 
As we will see in Section \ref{section-question1} the partition $\F= \{ f^{-1}(c)\}_{c>0}$ is forward  parallel.
Remark \ref{remark-transnormal-function} above implies that this partition is not a Finsler partition. 
This also follows from Lemma \ref{lemma-equidistant}  
because future spheres with center at $0$
(i.e., translation  of $\metriczermelo$-spheres in the direction of $W$ ) 
are not the same as the past spheres with the center at $0$ (i.e., translation  of $\metriczermelo$-spheres in opposite direction of $W$). 
As we have stressed in the introduction, this phenomenon is different from what happens in the Riemannian case, where transnormality already
implies that the level sets are equidistant. 

\end{example}

\begin{remark}
\label{remark-C2}
Let $(M,\zermelo)$ be a Randers space with Zermelo data $(\metriczermelo, W)$. \emph{ Let $f:M\to \mathbb{R}$ be a smooth 
$\metriczermelo$-transnormal function with $\tilde{\btransnormal}\in C^{2}\big(f(M)\big)$. Suppose also that
 $W$ is a $\F$-foliated vector field, where $\F= \{ f^{-1}(c)\}$. Using Lemma \ref{lemma-gradient-finsler-riemannian} it is possible to check that 
 $f$ is a $\zermelo$-transnormal function with  $\btransnormal \in C^{0}\big(f(M)\big)$}. 
As we are going to see below, there is a simple example where $\btransnormal\notin C^{2}\big(f(M)\big).$ 
This indicates another phenomenon that is different from the Riemannian case, where the assumption 
$\tilde{\btransnormal}\in C^{2}\big(f(M)\big)$ is natural.

\end{remark}

\begin{example}
\label{ex-C2}
Consider $f:D\to \mathbb{R}$ where $D$ is a disc of 
radius smaller than $1$ and  $f(x,y)=x^{2}+y^{2}$.  
Let $\zermelo$ be  the Randers metric with Zermelo data
$(\metriczermelo_0, W)$ where $\metriczermelo_0$ is the Euclidean metric of
$\mathbb{R}^{2}$ and $W=(x,y).$ 
From Lemma \ref{lemma-gradient-finsler-riemannian} we conclude that $b(t)=(2\sqrt{t}+2 t)^{2}$.
\end{example}

\section{ Question \ref{question-1}}
\label{section-question1}

The goal of this section is to give an alternative proof to  Proposition \ref{proposition-forward-equidistant} below, that was proved at
\cite{He-Yin-Shen}.

We start by recalling the next lemma, proved at \cite[Lemma 3.2.1]{Zhongmin-Shen}
\begin{lemma}
\label{lemma-transnormal-RF1} 
	Let  $(M,F)$ be a Finsler space, $U$ be an open subset of  $M$ and $f$ be a smooth function on 
$U$ without critical points on $U$.  
Set $\hat{g}:=g_{\nabla f}$ and $\widehat{F}:=\sqrt{\hat{g}}$. 
Then $$\nabla f=\widehat{\nabla}f,$$ where $\widehat{\nabla}f$ denotes the gradient of $f$ with respect to $\widehat{F}$. 
		Moreover $$F(\nabla f)=\widehat{F}(\widehat{\nabla} f).$$ 
\end{lemma} 

\begin{remark}
\label{remark-transnormal-FR}
As proved at \cite[Lemma 3.2.2]{Zhongmin-Shen},  the gradient $\nabla f$ of functions on Finsler space $(M,F)$
is orthogonal to each regular level set. 
\end{remark}

We also need this other known result, that follows by using a Koszul type formula associated to  the Chern connection. 

\begin{lemma}
\label{lemma-transnormal-RF2}
Let $X$ be a smooth vector field without singularities on
an open set $U$. Consider  the  Riemannian metric $\hat{g}:=g_{X}$ on $U$,
 the associated Riemannian connection (associated to $\hat{g}$)  $\widehat{\nabla}$ 
and the Chern connection $\nabla^{X}$. 
Then $\nabla_{X}^{X}X=\widehat{\nabla}_{X} X$.
In particular,  $X$ is a vector field on $U$  so that its integral curves are geodesics
(with respect to $F$) if and only if  $X$ has the same property
 with respect to  $\widehat{F}:=\sqrt{\hat{g}}.$
\end{lemma} 

\begin{proposition}
\label{proposition-forward-equidistant}
Let $(M,F)$ be a  forward complete Finsler space. 
Let $f:M\to \mathbb{R}$ be a $F$-transnormal function, $c<d$ regular values on $f(M)$
so that $[c,d]$ does not have singular values. Then for each $q\in f^{-1}(d)$
$$ d(f^{-1}(c),q)=d(f^{-1}(c),f^{-1}(d))=\int_{c}^{d}\frac{d s}{\sqrt{\btransnormal(s)} }. $$
In addition the integral curves of the vector field  $\nabla f$ (i.e., the gradient flow), 
when parameterized by arc length, are horizontal geodesics joining  $f^{-1}(c)$ to $f^{-1}(d)$
and realize the distance between these two regular leaves.  
 \end{proposition}
\begin{proof}
Set $U=f^{-1}([c,d])$. Then from Lemma \ref{lemma-transnormal-RF1} we conclude that
$$ \hat{g}(\widehat{\nabla}f,\widehat{\nabla}f)=\btransnormal\circ f $$
on $U$. In other words $f$ is also a transnormal function with respect to $\hat{g}$ (with the same $\btransnormal$). 
We are going to  use classical results about  Riemannian transnormal  function, recall
\cite{QMWang} and  \cite[Chapter 5]{AlexBettiol}.
Let $\alpha$ be an integral curve of $\widehat{\nabla} f$ starting at some point of $p\in f^{-1}(c)$
and $\beta$ its arc-lenght reparametrization. Then
\begin{itemize}
\item $\beta$ is a horizontal unit speed geodesic (with respect to $\hat{g}$),
\item $\beta|_{[0,r]}$ joins $f^{-1}(c)$ to $f^{-1}(d)$, where  $r=\int_{c}^{d}\frac{d s}{\sqrt{\btransnormal(s)} }$,
\item $\beta$ meets each regular level set just once. 
\end{itemize} 
From Lemmas \ref{lemma-transnormal-RF2} and \ref{lemma-transnormal-RF1} and Remark 
\ref{remark-transnormal-FR} we see that  $\beta$ is also the  arc-lenght reparametrization of the 
 integral curve of $\nabla f$ and also fulfills the  properties described above  for the Finsler metric $F$. 

Finally consider a segment of unit speed geodesic $\gamma$ joining $f^{-1}(c)$ to a point $q\in f^{-1}(d)$ 
realizing the distance between them. 
Then it is not difficult to see that $\gamma$ is contained in $U$, it meets $f^{-1}(c)$ just at one point and at this point the velocity of $\gamma$  has the same directions as  $\nabla f$. 
From the unicity   of geodesics we conclude that  $\gamma$  must coincide with one of the segments $\beta$ defined above and this conclude the proof.

\end{proof}

\begin{definition}
As we have seen above, given a transnormal function $f:M\to\mathbb{R}$, 
the integral curves of the vector field  $\nabla f$ (i.e., the gradient flow), 
when parameterized by arc-length is a geodesic. This segment of geodesic is called \emph{$f$-segment}. 
\end{definition}

\begin{remark}[Analyticity]
\label{remark-analyticity}
Assume that $f:M\to \mathbb{R}$ is an analytic function on an analytic manifold $M$. Then, as usual, local properties can be extended. For example assume that $f$ is a transnormal function in a neighborhood of a point $p$ of regular leaf $f^{-1}(c)$. 
Set $g(s,t)=f\big(\exp_{\beta(s)}(t\xi)\big)- f\big(\exp_{p}(t\xi) \big)$ where $\xi=\frac{\nabla f}{F(\nabla f)}$ and 
$s\to \beta(s)\in f^{-1}(c)$ is a curve such that $\beta(0)=p$. Note that  $g(s,t)=0$ for small $s$ and $t$ because $f$ is transnormal in a neighborhood of $p$. By analyticity of $f$ we conclude that the function $g$ is always zero, i.e., regular level sets are forward parallel. This and other quite similar straightforward arguments will be extensively used  in the next section.
\end{remark}


\section{ Question \ref{question-2}  and proof of Theorem   \ref{finsler-partition} }
\label{section-question2}

Let us first sketch the idea of the proof of Theorem \ref{finsler-partition}.  
First we are going to show that there exists a neighborhood $U_0$ of $f^{-1}(b)$ so that $\F$ restricted to $U_{0}\setminus f^{-1}(b)$ 
is a Finsler foliation, see Lemma \ref{lemma-neighborhood-finsler}.
This will be proved using the  analyticity of $f$, the fact
that regular level sets, future and past cylinder have codimension 1 and Lemma \ref{lemma-equidistant}.
Once we have assured that $\F$ is a Finsler foliation on $U_{0}\setminus f^{-1}(b)$, we will apply index-Morse arguments from \cite[Theorem 5.63]{AlexBettiol} to conclude that $f^{-1}(b)$ is in fact a submanifold, see Lemma \ref{lemma-critical-levelsets-submanifold}.  
Finally analyticity  will allow us to extend the property \emph{of being a Finsler partition on $U_{0}$}  to whole  $M$. 

Now let us give a few more details about the proof through the next two lemmas and  a series of claims.

\begin{lemma}
\label{lemma-neighborhood-finsler}
There exists a neighborhood $U_{0}$ of the critical level set  $f^{-1}(b)$  where $\F$ 
fulfills the following propery: if $x\in U_{0}\setminus f^{-1}(b)$ and $\gamma$ 
is a geodesic so that $\gamma(0)=x$ and $\gamma'(0)$ is orthogonal to the level set that contains $x$, 
then $\gamma$ is orthogonal to all regular level sets of  $M$ it meets. 
\end{lemma}
\begin{proof}

\begin{claim}
\label{claim0-main}
  $C^{-}_{r_{c}^{-}}\big( f^{-1}(b)\big)=f^{-1}(c)$, for each $c<b$ 
and  $r_{c}^{-}=d(f^{-1}(c),f^{-1}(b))$.
\end{claim}

In fact let $x_{0}\in f^{-1}(b)$ be a point so that $d(f^{-1}(c),x_{0})=d(f^{-1}(c),f^{-1}(b))=r_{c}^{-}$. 
Let $\gamma:[0,r_{c}^{-}]\to M$ be a unit speed geodesic  so that $\gamma(0)\in f^{-1}(c)$ and $\gamma(r_{c}^{-})=x_{0}$. 
Note that  $\gamma$  is an extension of an $f$-segment and minimize the distance. These facts and the analyticity imply that  
each $f$-segment  starting at $f^{-1}(c)$  meets $f^{-1}(b)$ at the first time at 
$t=r_{c}^{-}$ and this implies that $f^{-1}(c)\subset C^{-}_{r_{c}^{-}}\big( f^{-1}(b)\big)$.
Now consider $x\in C^{-}_{r_{c}^{-}}\big( f^{-1}(b)\big)$ and $\tilde{c}=f(x).$
From  what we have discussed  before we have that $f^{-1}(\tilde{c})\subset C^{-}_{r_{\tilde{c}}^{-}}\big( f^{-1}(b)\big)$. Therefore
$r_{c}^{-}=r_{\tilde{c}}^{-}=r$. 
 Assume by contradiction, that $f(x)=\tilde{c}<c$. Let $\gamma:[0,r]\to M$ be a minimal unit speed geodesic 
 joining $\gamma(0)\in f^{-1}(\tilde{c})$ to $\gamma(r)\in f^{-1}(b)$. The fact that the regular leaves have codimension one allows us to conclude that $\gamma$  is an $f$-segment and cross $f^{-1}(c)$ at time $t<r$ what is a contradiction with the fact that 
$f^{-1}(c)\subset C^{-}_{r}\big( f^{-1}(b)\big)$. A similar contradiction happens if one supposes  that $f(x)=\tilde{c}>c$. Therefore
$f(x)=\tilde{c}=c$ i.e., $C^{-}_{r_{c}^{-}}\big( f^{-1}(b)\big)\subset f^{-1}(c)$ and this concludes the proof of Claim \ref{claim0-main}.

From Lojasiewicz's  Theorem (recall \cite[Theorem 6.3.3]{Krantz-Parks}) we know that  
the level set $f^{-1}(b)$ is  stratified into  submanifolds. 
Let $\Sigma$ denote a (connected) stratum with local larger  dimension, i.e., if $x\in \Sigma$ then there is a neighborhood $U$ of $x$
so that the only components of $f^{-1}(b)\cap U$ are components of  $\Sigma.$   
For each $x_{\alpha}\in \Sigma\subset f^{-1}(b)$ consider  a relatively compact neighborhood 
$P_{\alpha}\subset \Sigma$ of $x_\alpha$ so that $\overline{P_{\alpha}}$ is in the interior of $\Sigma$ and 
$C^{-}_{r_{c}^{-}}\big( P_{\alpha} \big)=C^{-}_{r_{c}^{-}}\big( f^{-1}(b)\big)\cap U$ for some neighborhood $U$
of $x_{\alpha}$ and for $c$ close to $b$. By using Claim \ref{claim0-main} we infer the next claim.

\begin{claim}
\label{claim1-main}
For each $c$ close to $b$ the past cylinder 
$C^{-}_{r_{c}^{-}}\big( P_{\alpha} \big)$ is an open set of $f^{-1}(c).$
\end{claim}

The above claim and the   analyticity of $f$ imply:
\begin{claim}
\label{claim2-main}
let $\gamma_{u_{i}}$  be the unit speed geodesic 
with $\gamma_{u_{i}}'(0)=u_{i}$, for   $u_{1}, u_{2} \in \nu^{1}(P_{\alpha})$.  
Then
\begin{enumerate}
\item[(a)] $f\big(\gamma_{u_{1}}(t)\big)=f\big(\gamma_{u_{2}}(t)\big)$, for
$t\in\mathbb{R}$
\item[(b)] $\gamma_{u_{i}}$ is orthogonal to each regular level
set of $f$. 
\end{enumerate}
\end{claim}

Let $\gamma$ be a unit speed geodesic  orthogonal
to $P_{\alpha}$. It is not difficult to see that there exists a $c_0$ so that for each $c\in[c_{0},b)$ there
exists   $r_{c}^{-}>0$ and $r_{c}^{+}>0$  so that
$f\big(\gamma(-r_{c}^{-})\big)=c
=f\big(\gamma(r_{c}^{+}) \big).$
From Claim \ref{claim2-main} one can infer that, for each other unit speed geodesic $\gamma_{u}$  orthogonal 
to $P_{\alpha}$, we have that $\gamma_{u}(r_{c}^{+})\in f^{-1}(c)$, and hence $C_{r_{c}^{+}}^{+} (P_{\alpha})\subset f^{-1}(c)$. This fact and the fact that 
  $C_{r_{c}^{+}}^{+} (P_{\alpha})$ and $f^{-1}(c)$ have  codimension 1  imply that 
\begin{claim}
\label{claim3-main} There exists $c_0$ so that 
$C_{r_{c}^{+}}^{+} (P_{\alpha})$ is an open set of $f^{-1}(c)$
for each $c\in [c_{0},b)$.
\end{claim}

Lemma \ref{lemma-equidistant}, Claims \ref{claim1-main} and \ref{claim3-main} imply the next claim. 

\begin{claim}
\label{claim4-main}
There exists a neighborhood $U_{\alpha}$ of $P_{\alpha}$  where $\F$ 
fulfills the following propery: if $x\in U_{\alpha}\setminus P_{\alpha}$ and $\gamma$ 
is a geodesic so that $\gamma(0)=x$ and $\gamma'(0)$ is orthogonal to the level set that contains $x$, 
then $\gamma$ is orthogonal to all regular level sets of  $M$ it meets. 

\end{claim}

Let $U_0$ be the saturation of $U_{\alpha}$. Claim \ref{claim4-main} and analyticity of $f$ imply that  
$U_{0} \setminus f^{-1}(b)$ also fulfills the property of 
Claim \ref{claim4-main} and in particular $\F$ restricted to $U_{0}\setminus f^{-1}(b)$ is a Finsler foliation, as we wanted to prove. 

\end{proof}

\begin{lemma}
\label{lemma-critical-levelsets-submanifold}
$f^{-1}(b)$ is an embedded submanifold.
\end{lemma}
\begin{proof}

Let $\eta_{t\xi}:f^{-1}(c)\to M$ be the map defined as $\eta_{t\xi}(x)=\exp_{x}(t\xi)$ where 
$\xi=\frac{\nabla f}{F(\nabla f)}$. We will call this kind of map as an \emph{end point map pointing in the direction of $\xi$.}

\begin{claim}
\label{claim-5main}
There exists an $\epsilon>0$ so that  $\eta_{t\xi}:f^{-1}(c)\to f^{-1}(d)$  
is a diffeomorphism between regular level sets, for each $t\in (0,r_{c}^{-}) \cup (r_{c}^{-}, r_{c}^{-} +\epsilon)$ 
and $c$ close to $b$ (e.g., $c>c_{0}$).
\end{claim}	
In fact from  analyticity it is easy to see that  $\eta_{t\xi}:f^{-1}(c)\to f^{-1}(d)$.  
In order to prove that it is a diffeomorphism, it suffices to 
construct the smooth inverse. Let $\tilde{\xi}$ be the normal vector field along $f^{-1}(d)$ pointing in the 
opposite (resp. same) direction of 
$\nabla f$ if $t\in (r_{c}^{-}, r_{c}^{-} +\epsilon)$ (resp. if $t\in (0,r_{c}^{-})$). 
Define $t\to \gamma(t)$ as $\gamma(t)=\exp(t\tilde{\xi})$ and define the \emph{end past map}
 $\eta_{t}^{-}:f^{-1}(d)\to M$ as $\eta_{t}^{-} (x):=\gamma(-t) $. 
Analyticity and  Lemma \ref{lemma-neighborhood-finsler} imply that the map   $\eta_{t}^{-}:f^{-1}(d)\to f^{-1}(c)$
 is the inverse of $\eta_{t\xi}$.

	\begin{claim}
	\label{claim-6main}
 		The derivative of map $\eta_{r^-_{c}\xi}:f^{-1}(c)\longrightarrow M$ has constant rank. 
 	\end{claim}
The idea of the proof is based on \cite[Theorem 5.63]{AlexBettiol}. 
Let us briefly recall it, accepting results on Jacobi field on Finsler spaces; 
see \cite{Javaloyes-Soares} and \cite{I-R-Peter}. 
For $p\in f^{-1}(c)$, consider the geodesic  $t\to \gamma_{p}(t)=\exp_p(t\xi)$. 
Since $f^{-1}(c)$ is a hypersurface, we can infer that the point 
$\gamma_{p}(t)$ is a $f^{-1}(c)$-focal point of multiplicity $k$ 
if and only if $p$ is a critical point of $\eta_{t\xi}$ and $\dim\ker d(\eta_{t\xi})_p=k$. 
Furthermore, for the appropriate choice of $\epsilon>0$,    
  Claim \ref{claim-5main} implies that if  $t\in I=[0,r^-_{c}+\epsilon]$, 
then  $\eta_{t\xi}$  may only fail to be an immersion if $t=r^-_{c}.$
These two facts together imply that for every $x\in f^{-1}(c),$ 
$$m(\gamma_x)=\dim\ker d(\eta_{r^-_{c}\xi})_{x},$$
where $m(\gamma_x)$ denotes the number of focal points on $\gamma_{x}$ counted with multiplicities on
$\gamma_{x}|_{I}$.
 From Morse Index  $$m(\gamma_p)\leq m(\gamma_x),$$ for  $x\in f^{-1}(c)$ near to $p$. 
Since $$\dim\ker d(\eta_{r^-_{c}\xi})_p\geq \dim\ker d(\eta_{r^-_{c}\xi})_x,$$ for  $x\in f^{-1}(c)$,
we conclude that  
$$\dim\ker d(\eta_{r^-_{c}\xi})_p=  \dim\ker d(\eta_{r^-_{c}\xi})_x,$$ for $x$ near to $p$.
This and the  connectivity of $f^{-1}(c)$ finish  the proof of Claim  \ref{claim-6main}.

\begin{claim}
\label{claim-7main}
The  map $\eta_{r^-_{c}\xi}:f^{-1}(c)\rightarrow f^{-1}(b)$ is  surjective and $f^{-1}(b)$ is an immersed submanifold.  
\end{claim}
In fact, from Claim \ref{claim-6main} and Claim \ref{claim1-main} we infer that $\rank d \eta_{r^-_{c}\xi}=\dim \Sigma$. 
This fact, Claim \ref{claim1-main}, connectivity argument and definition of stratification imply that
$\overline{\Sigma}\subset \eta_{r^-_{c}\xi}\big(f^{-1}(c)\big).$ Note that this also holds for each other stratum of $f^{-1}(b)$ 
with local larger  dimension. Therefore from definition of stratification we conclude that the map
$\eta_{r^-_{c}\xi}:f^{-1}(c)\rightarrow f^{-1}(b)$ is surjective.  
From rank theorem we deduce  that   $f^{-1}(b)$ is an immersed submanifold (with possible intersections).

Claims \ref{claim0-main} and \ref{claim-7main} imply that $f^{-1}(b)$ is an embedded submanifold, as we wanted to prove.

\end{proof}


Now we want to extend the property of Lemma \ref{lemma-neighborhood-finsler} to whole analytic manifold $M$. 

 Let $\gamma:[0,r]\to M$ be a fixed segment of geodesic  joining $f^{-1}(b)$ to $f^{-1}(a)$  so that 
 $r=d(f^{-1}(b), f^{-1}(a))$.

Consider a partition $0=t_{0}<t_{1}<\cdots < t_{n-1}< t_{n}=r$, a finite covering of $[0,r]$ by open intervals $I_{i}$ ($i=0\cdots n$) 
centered at $t_{i}$  such that $U_{i}:=f^{-1}(f(\gamma(I_{i})))$ is an open neighborhood of $f^{-1}(c_{i})$ (where $c_{i}=f(\gamma(t_{i}))$)
contained in the  future and past neighborhoods of $f^{-1}(c_{i})$.
Finally consider $\{ s_{i}\}_{i=0}^{n-1}$ so that  $0<s_{0}<t_{1}<s_{1}< t_{2}< s_{2} \cdots < s_{n-1}< t_{n}= r$  and $s_{i}\in I_{i}\cap I_{i+1}.$ Set 
$\tilde{c}_{i}=f(\gamma(s_{i}))$. Note that $f^{-1}(\tilde{c}_{0})$ is contained in the neighborhood $U_{0}\cap U_{1}$.
Therefore Lemma \ref{lemma-neighborhood-finsler} allows us to infer that the geodesics (starting at 
$f^{-1}(\tilde{c}_{0})$ pointing in the opposite direction of the gradient) arise orthogonally to $f^{-1}(c_{1})$.
Hence $f^{-1}(\tilde{c}_{0})$ is contained in the connected component $C^{-}$ 
of the  past cylinder $C^{-}_{r}(f^{-1}(c_{1}))$ of axis $f^{-1}(c_{1})$.
Therefore, since both have the same dimension, they coincide. 
On the other hand, the end point map
$\eta_{r\xi}: C^{-} \to f^{-1}(c_{1}) $ is a diffeomorphism, where 
$\xi$ is the unit normal vector along $ C^{-}$ pointing in the opposite direction of the gradient.
  Similarly, end point maps induce diffeomorphisms between $f^{-1}(c_{1})$ and (connected components of)  its future cylinders.  
These  facts together imply that  Lemma \ref{lemma-neighborhood-finsler} also holds in a neighborhood of $f^{-1}(\tilde{c}_{1})$. 
By induction we infer that Lemma \ref{lemma-neighborhood-finsler} is true in a neighborhood of $f^{-1}(\tilde{c}_{n-1})$.
Following the same proof of Lemmas \ref{lemma-neighborhood-finsler} and \ref{lemma-critical-levelsets-submanifold}
we conclude that Lemma \ref{lemma-neighborhood-finsler} holds in a neighborhood $U_{n}\setminus f^{-1}(a)$ and that the level
set $f^{-1}(a)$ is an embedded submanifold, finishing the proof of the theorem.


\section{Question \ref{question-3} } 
\label{section-question3}

In this section we approach Question \ref{question-3} and prove the next proposition.

\begin{proposition}
\label{transnormal-SRF}
		Let $(M,F)$ be a compact, connected and smooth Finsler manifold and 
		$f:M\to [a,b]$ be a smooth $F$-transnormal function with $F^2(\nabla f)=\mathfrak{b}(f)$, 
		where $\mathfrak{b}$ is a $C^1$ function on $[a,b]$. Suppose that: 
	\begin{itemize}
	  \item[(a)] the level sets are connected,  
		\item [(b)] the critical level sets $f^{-1}(a)$ and $f^{-1}(b)$ are submanifolds of  codimension greater than one.
	\item [(c)] $a$ and $b$ are the only singular values of $[a,b]$, 
	\item [(d)] $\mathfrak{b}'(a)\neq 0\neq \mathfrak{b}'(b).$
	\end{itemize}
Then there exists a Riemannian metric on $M$ such that $\mathcal{F}=\{f^{-1}(c)\}_{c\in [a,b]}$ is a singular Riemannian foliation. 
\end{proposition}
\begin{remark}
As discussed by Wang \cite{QMWang}, 
conditions (c) and (d) above are satisfied by a 
Riemannian transnormal function
if  $\btransnormal \in C^{2}[a,b]$, and these are important conditions 
e.g, there exist examples of (Riemannian) transnormal functions where (d) is not satisfied and
the level sets  of $f$ are not even leaves of a   singular  foliations. 
The problem in the Finsler case is that the assumption   $\btransnormal \in C^{2}[a,b]$  
 seems  to be too strong, recall Remark \ref{remark-C2} and Example \ref{ex-C2}.
Therefore it remains for us to assume (c) and (d) as hypotheses.
 Note that one can even ask  if the smoothness of $f$ and the assumption that 
$\btransnormal \in C^{2}[a,b]$ already imply some property about the  Finsler metric $F$. For example one can ask: \emph{is the Finsler metric $F$ already  Riemannian (or Riemannian in  transversal directions to the singularities) or at least  reversible (in transversal directions to the singularities) when the function  $\btransnormal$  is  $C^{2}$?} 
In particular, it would be natural to try to   establish an analogy between  this question and the well 
 known fact that if $\exp_p$ is  $C^{2}$  at zero then  $F$ is a Riemannian  metric. 
\end{remark}

In order to prove the above proposition we will use a result about Bott-Morse functions. 
Let $f:M\to \mathbb{R}$ be a smooth function. As usual we can define the Hessian of $f$ at a critical point $p\in M$ 
as the symmetric linear operator $\mathrm{Hess}f_{p}:T_{p}M\times T_{p}M\to \mathbb{R}$ defined by    
 $\mathrm{Hess} f_{p}(v,w)=\vec{v}_{p}\vec{w} f$, where $\vec{v}, \vec{w}$ are extentions of $v$ and $v$, resp.  
Let $Cr(f)$ denote the critical level set of $f$. Recall that
$f$ is called a \emph{Morse-Bott function} if $Cr(f)$ is union of connected submanifolds and 
the $\mathrm{kern}\, \mathrm{Hess} f$ of singular points  coincides with the tangent spaces of $Cr(f)$.
In particular if $g$ is some Riemannian metric on $M$ and $S$ is a submanifold normal to 
$Cr(f)$ at $p$ then $\mathrm{Hess}_{p}$ restricted to $T_{p}S$ turns to
be non degenerate.

The next strong result stresses the  relation between Bott-Morse functions and (Riemannian) transnormal functions.

\begin{theorem}[\cite{qian2015isoparametric}]
\label{thm-morse-bott}
	Let $M$ be a compact smooth manifold, and $f:M\to \mathbb{R}$ a \textit{Morse-Bott} function with $Cr(f)=M_+\sqcup M_-$, where $M_+$ and $M_-$ are both closed connected submanifolds of codimensions bigger than 1. Then there exists a Riemannian metric on $M$ so that $f$ is transnormal. In fact, the metric can be chosen so that $M_+$ and $M_-$ are both totally geodesics. 
\end{theorem}

Our goal is to check that the $F$-transnormal function that satisfies the hyphothesis of Proposition \ref{transnormal-SRF} is a Bott-Morse function. Once we have proved this, our result will follow directly  from Theorem \ref{thm-morse-bott} and Wang \cite{QMWang}.

We start by recalling the definition of Finslerian $\mathrm{Hess}^{F}$ on a Finsler manifold $(M,F)$ on non singular values of $f$. 

\begin{definition}\label{hes.reg}
	Let $f:(M,F)\to \mathbb{R} $ be a smooth function on a 
	Finsler manifold $M$ and $ U=\{x\in M, \,  d f_x\neq0\}$. 
	We define $\mathrm{Hess}^{F} f$   on $U$ as 
	\begin{equation*}
	\begin{aligned}
	\mathrm{Hess}^{F} f:&\mathfrak{X}(\mathcal{U})\times \mathfrak{X}(\mathcal{U})\to\mathbb{R}\\
	&(Y,X)\mapsto g_{\nabla f}(\nabla_{Y}^{\nabla f}\nabla f,X)
	\end{aligned}
	\end{equation*}
\end{definition}
\begin{lemma}
\label{lemma-hessf-1}
		Let $f:(M,F)\to \mathbb{R} $ be a smooth function on a Finsler manifold $M$ and $\mathcal{U}=\{x\in M , \, d f_x\neq0\}$. Then 
	$	\mathrm{Hess}^{F}_{x}(Y,X)=Y(X(f))_{x} -df_{x}(\nabla_Y^{\nabla f}X).$
	\end{lemma}
	\begin{proof}
		By the almost $g$-compatibility of the connection and the definition of the gradient we have 
		\begin{eqnarray*}
		Y(X(f))& = & Y(df X)  =  Y (g_{\nabla f}(\nabla f,X)) \\
					& = & g_{\nabla f}(\nabla_{Y}^{\nabla f}\nabla f,X)+g_{\nabla f}(\nabla_{Y}^{\nabla f}X,\nabla f)\\
					&+ & 2C_{\nabla f}(\nabla_{Y}^{\nabla f}\nabla f,\nabla f,X).
		\end{eqnarray*}
		From Eq.~(\ref{eq-properties-Cartan-tensor}) 
		we conclude that $C_{\nabla f}(\nabla_{Y}^{\nabla f}\nabla f,\nabla f,X)=0$. This
		fact and the definition of gradient imply the lemma. 
	\end{proof}

\begin{lemma}
\label{lemma-hessf-2}
	Let $f:(M,F)\to \mathbb{R} $ be a smooth function which is $F$-transnormal with 
	$F^2(\nabla f)=\mathfrak{b}\circ f$ and $U=\{x\in
	M , \, d f_x\neq0\}$. Then on $U$ we have 
	$$  \mathrm{Hess}^{F} f(\nabla f,\nabla f)=\frac{1}{2}\mathfrak{b}'(f)\mathfrak{b}(f).$$ 
	In particular, one can write $ \mathrm{Hess}^{F} f(\frac{\nabla f}{F(\nabla f)},\frac{\nabla f}{F(\nabla f)})=\mathfrak{b}'(f).$
\end{lemma}
\begin{proof}
	The definition of transnormal function, and Eq.~(\ref{eq-properties-Cartan-tensor}) imply

	\begin{eqnarray*}
	\mathrm{Hess}^{F} f(\nabla f,\nabla f)& =& g_{\nabla f}(\nabla^{\nabla f}_{\nabla f}\nabla f,\nabla f)  =  \frac{1}{2}\nabla f \big(g_{\nabla f}(\nabla f,\nabla f)\big)\\
	& = & \frac{\nabla f}{2} (F^2(\nabla f)) =  \frac{\nabla f}{2}(\mathfrak{b}\circ f)\\
	& = & \frac{1}{2} \mathfrak{b}'(f)df(\nabla f) =  \frac{1}{2}\mathfrak{b}'(f)g_{\nabla f}(\nabla f,\nabla f)\\
	& = & \frac{1}{2} \mathfrak{b}'(f)F^2(\nabla f) =  \frac{1}{2} \mathfrak{b}'(f)\mathfrak{b}(f). 
	\end{eqnarray*}
	\end{proof}

Consider an arbitrary metric $g$ and a slice $S$ orthogonal to $f^{-1}(b)$ at $p$. Let $X$ be a vector of 
$T_{p}S$ and consider the (only) vector $V$ at $\nu_{q}(f^{-1}(b))$, that projects to $X$, i.e., so that
$V=X+X^{v}$ where $X^{v}$ is tangent to $f^{-1}(b)$, see  \cite[Lemma 2.9]{AlexAlvesJavaloyes}. 
 Let $\gamma$ be the unit speed geodesic that contains the integral lines of $\nabla f$   so that
$p=\lambda\gamma'(0)=V$, for the apropriate $\lambda\neq 0$. Lemmas \ref{lemma-hessf-1} and \ref{lemma-hessf-2} imply 
\begin{equation}
\label{eq-hess-alpha}
\mathrm{Hess}f(\gamma'(0),\gamma'(0)) =\frac{1}{2}\mathfrak{b}'(b).
\end{equation}

 Eq. \eqref{eq-hess-alpha}  and the fact that $T_{p}f^{-1}(b)=\ker \mathrm{Hess}f_{p}$ imply that 
$$\mathrm{Hess}f_{p}(X,X) =\frac{\lambda^{2}}{2}\mathfrak{b}'(b)\neq 0.$$
The last equation and the arbitrary choice of $X\in T_{p}S$ 
imply   $\mathrm{Hess}f_{p}$ is non degenerate  at $T_{p}S$. A similar proof is valid for $f^{-1}(a)$ and hence 
$f$ is a Morse Bott-function, as we wanted to prove. 


\bibliographystyle{amsplain}

\end{document}